\documentclass[12pt,a4paper]{article}

\input{mathdefs.sty}
\input{jens_defs.sty}

\renewcommand{\myboldfont}{\mathbb}

\usepackage[left=2.4cm, right=2.4cm , top=2.4cm, bottom=3.5cm]{geometry}

\usepackage{bm}
\usepackage{enumerate}

\newtheorem{theorem}{Theorem}
\newtheorem{lemma}[theorem]{Lemma}

\newtheorem{prop}[theorem]{Proposition}

\newcounter{theremark}
\newenvironment{remark}{\medskip\parindent=0pt\textbf{Remark \arabic{theremark}\addtocounter{theremark}{1}.\ }\ignorespaces}{\medskip\par\normalsize\parindent=18pt}

\numberwithin{equation}{section}
\newcommand*\samethanks[1][\value{footnote}]{\footnotemark[#1]}

\usepackage{graphicx}

\usepackage[font=small,labelfont=bf]{caption}

\makeatletter

\renewcommand{\section}{\@startsection{section}{0}{0pt}{3.6ex plus 0.2ex minus 0.1ex}{2.3ex plus 0.1ex minus 0.1ex}{\center\normalfont\sc\large}}

\makeatother

\begin{document}

\title{The two-point correlation function of the fractional parts of $\sqrt{n}$ is Poisson\thanks{Research supported by ERC Advanced Grant HFAKT. J.M. is also supported by a Royal Society Wolfson Research Merit Award.}}
\author{Daniel El-Baz\thanks{School of Mathematics, University of Bristol, Bristol BS8 1TW, U.K.} \and Jens Marklof\samethanks \and Ilya Vinogradov\samethanks}
\date{\today}

\maketitle

\begin{abstract}
Elkies and McMullen [Duke Math.~J.~123 (2004) 95--139] have shown that the gaps between the fractional parts of $\sqrt n$ for $n=1,\ldots,N$, have a limit distribution as $N$ tends to infinity. The limit distribution is non-standard and differs distinctly from the exponential distribution expected for independent, uniformly distributed random variables on the unit interval. We complement this result by proving that the two-point correlation function of the above sequence converges to a limit, which in fact coincides with the answer for independent random variables. We also establish the convergence of moments for the probability of finding $r$ points in a randomly shifted interval of size $1/N$. The key ingredient in the proofs is a non-divergence estimate for translates of certain non-linear horocycles. 
\end{abstract}

\section{Introduction\label{sec:intro}}

It is well known that, for every fixed $0<\alpha<1$, the fractional parts of $n^\alpha$ ($n=1,\ldots,N$) are,  in the limit of large $N$, uniformly distributed mod 1.  Numerical experiments suggest that the gaps in this sequence converge to an exponential distribution as $N\to\infty$, which is the distribution of waiting times in a Poisson process, cf.~Fig.~\ref{fig1}. The only known exception is the case $\alpha=1/2$. Here Elkies and McMullen \cite{ElkiesMcM04} proved that the limiting gap distribution exists and is given by a piecewise analytic function with a power-law tail (Fig.~\ref{fig2}). In the present study we show that a closely related local statistics, the two-point correlation function, has a limit which in fact is consistent with the Poisson process, see Fig.~\ref{fig3}. The proof of this claim follows closely our discussion in \cite{EMV_directions_2013}, which produced an analogous result for the two-point statistics of directions in an affine lattice. (We note that Sinai \cite{Sinai13} has recently proposed an alternative approach to the statistics of $\sqrt n\bmod 1$, but will not exploit this here.)

\begin{figure}
\begin{centering}\includegraphics[width=0.9\textwidth]{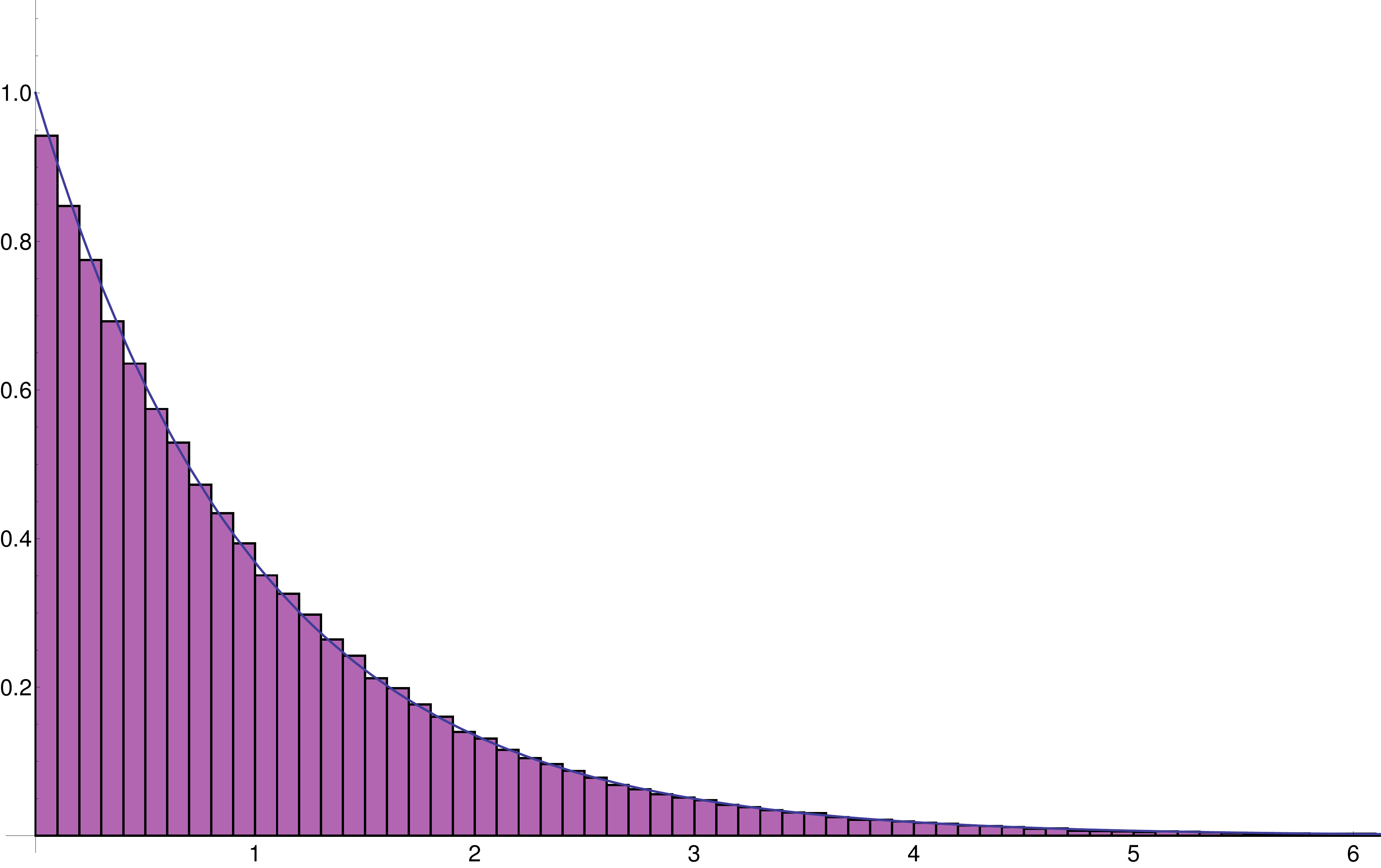}\par\end{centering}
\caption{Gap distribution of the fractional parts of $n^{1/3}$ with $n\leq 2\times 10^5$.  \label{fig1}}
\end{figure}
\begin{figure}
\begin{centering}\includegraphics[width=0.9\textwidth]{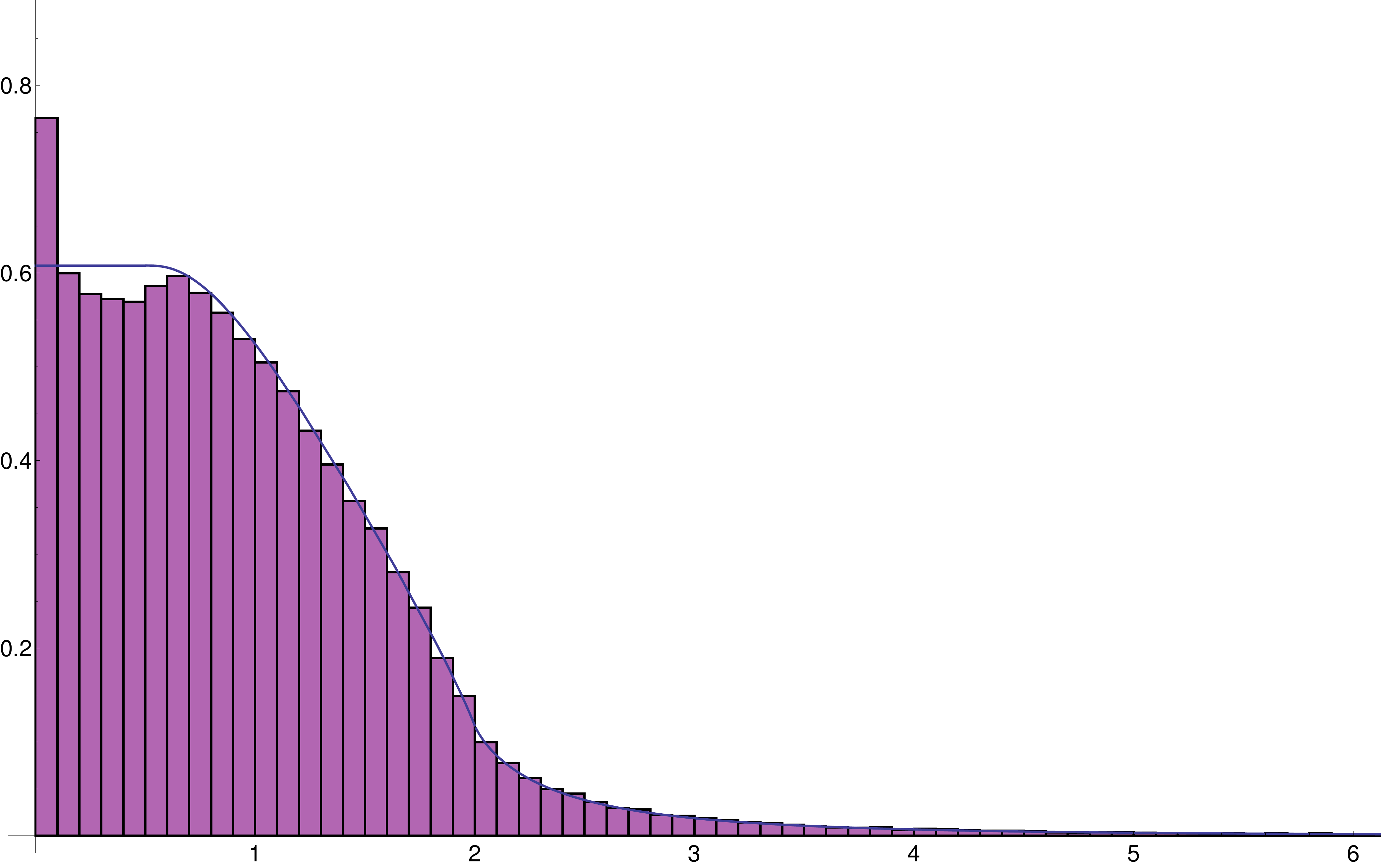}\par\end{centering}
\caption{Gap distribution of the fractional parts of $\sqrt n$ with $n\leq 2\times 10^5$.  \label{fig2}}
\end{figure}
\begin{figure}
\begin{centering}\includegraphics[width=0.9\textwidth]{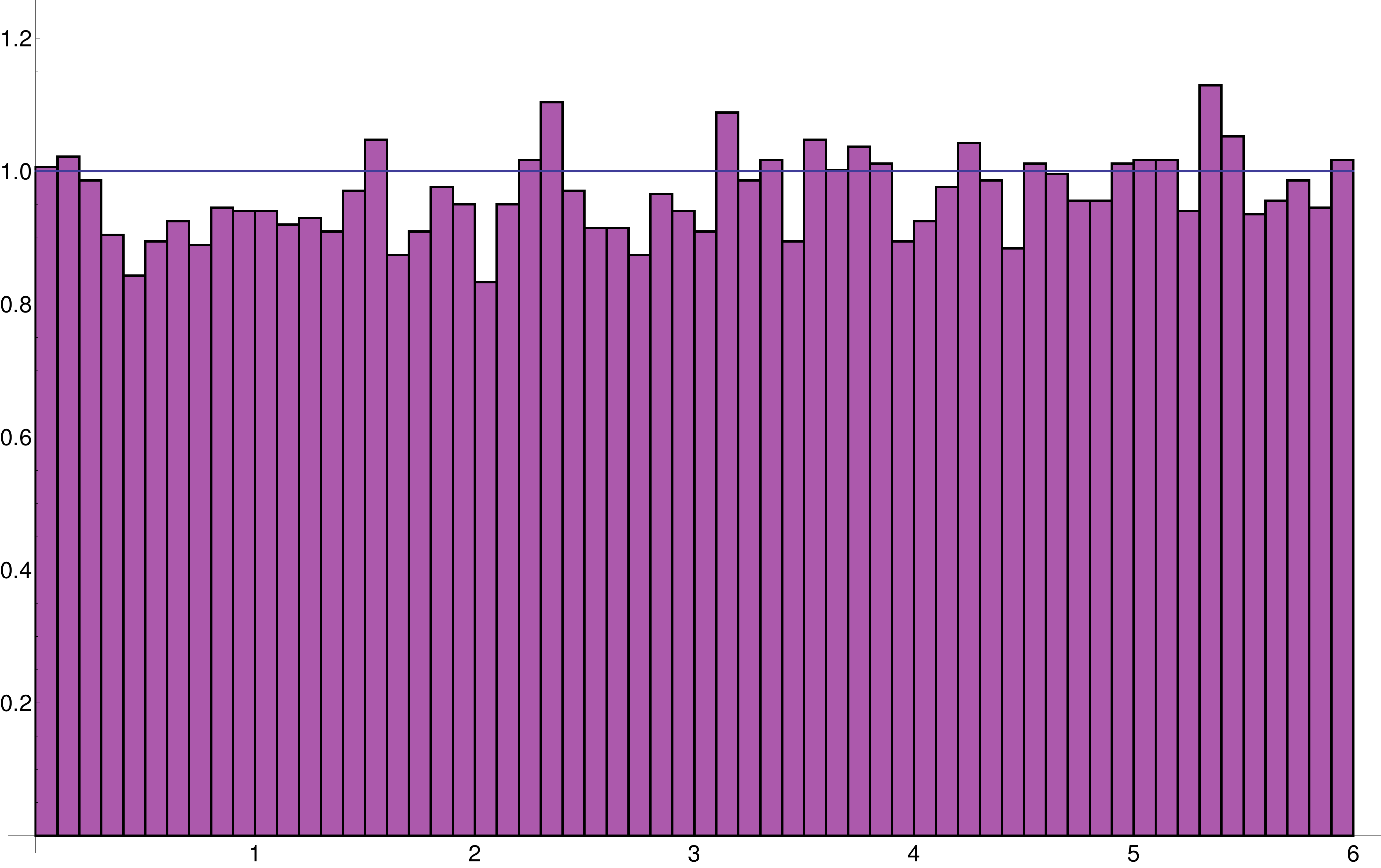}\par\end{centering}
\caption{Two-point correlations of the fractional parts of $\sqrt n$ with $n\leq 2000$, $n\notin\Box$.  \label{fig3}}
\end{figure}

Other number-theoretic sequences, whose two-point correlations are Poisson, include the values of positive definite quadratic forms subject to certain diophantine conditions \cite{sarnak_values_1997, Eskin05quadraticforms}, forms in more variables \cite{vanderkam_values_1999,vanderkam_pair_correlation_1999,VanderKam00}, inhomogeneous forms in two \cite{marklof_pair_correlation_2003,margulis_quantitative_2011} and more variables \cite{marklof_pair_correlation_II_2002}, and the fractional parts of $n^2\alpha$ (and higher polynomials) for almost all $\alpha$ \cite{rudnick_pair_1998,marklof_equidistribution_2003,heath-brown_pair_2010} (specific examples, e.g.~$\alpha=\sqrt 2$ are still open). 

To describe our results, let us first note that $\sqrt{n}=0\bmod 1$ if and only if $n$ is a perfect square. We will remove this trivial subsequence and consider the set
 \begin{equation}
 \scrP_T=\{\sqrt n \bmod 1\colon 1\le n\le T,\; n\notin\Box\}\subset \T:=\R/\Z
 \end{equation}
 where $\Box\subset\N$ denotes the set of perfect squares. The cardinality of $\scrP_T$ is $N(T)=T-\lfloor \sqrt T\rfloor$. We label the elements of $\scrP_T$ by $\alpha_1,\ldots,\alpha_{N(T)}$. The pair correlation density of the $\alpha_j$ is defined by
 \begin{equation}
R_N^2(f) = \frac{1}{N} \sum_{m\in\Z} \sum_{\substack{i,j=1\\ i\neq j}}^N f\big(N(\alpha_i-\alpha_j+m)\big),
\end{equation}
where $f\in C_0(\R)$ (continuous with compact support). Note that $R_N^2$ is not a probability density.
Our first result is the following.

\begin{theorem}\label{thm0}
For any $f\in C_0(\R)$,
\begin{equation}\label{paircon}
\lim_{T\to\infty}  R_{N(T)}^2(f) = \int_{\R} f(s)\, ds.
\end{equation}
\end{theorem}

That is, $R_{N(T)}^2$ converges weakly to the two-point density of a Poisson process. 

Both the convergence of the gap distribution and of the two-point correlations follow from a more general statistics, the probability of finding $r$ elements $\alpha_j$ in randomly placed intervals of size proportional to $1/N(T)$. Given a bounded interval $I \subset \R$, define the subinterval $J=J_N(I,\alpha)=N^{-1} I +\alpha+\Z\subset\T$ of length $N^{-1}|I|$, and let
\begin{equation}
\scrN_{T}(I,\alpha) = \# \scrP_T\cap J_{N(T)}(I,\alpha) .
\end{equation}

It is proved in \cite{ElkiesMcM04} that, for $\alpha$ uniformly distributed in $\T$ with respect to Lebesgue measure $\lambda$, the random variable $\scrN_{T}(I,\alpha)$ has a limit distribution $E(k,I)$. That is, for every $k\in\Z_{\geq 0}$,
\begin{equation}\label{one}
\lim_{T\to\infty}\lambda(\{ \alpha\in\T : \scrN_{T}(I,\alpha) = k\}) = E(k,I).
\end{equation}
As Elkies and McMullen point out, these results hold in fact for several test intervals $I_1,\ldots, I_m$:

\begin{theorem}[Elkies and McMullen \cite{ElkiesMcM04}] \label{th:prelim}
Let $I=I_1\times\cdots \times I_m\subset \R^m$ be a bounded box. Then there is a probability distribution $E(\,\cdot\,,I)$ on $\Z_{\geq 0}^m$ such that, for any $\veck=(k_1,\ldots,k_m)\in\Z_{\geq 0}^m$
\begin{equation}
\lim_{T\to\infty}\lambda(\{ \alpha\in\T : \scrN_{T}(I_1,\alpha) = k_1,\ldots,\scrN_{T}(I_m,\alpha) = k_m\}) = E(\veck,I).
\end{equation}
\end{theorem}

The limiting point process characterised by the probabilities $E(\veck,I)$ is the same as for the directions of affine lattice points with irrational shift \cite{marklof_strombergsson_free_path_length_2010,EMV_directions_2013}; in the notation of \cite{EMV_directions_2013}, $E(\veck,I)=E_0(\veck,I)=E_{0,\vecxi}(\veck,I)$ with $\vecxi\notin\Q^2$. This process is described in terms of a random variable in the space of random affine lattices, and is in particular not a Poisson process. The second moments and two-point correlation function however coincide with those of a Poisson process with intensity $1$. This is a consequence of the Siegel integral formula, see \cite{EMV_directions_2013}. Specifically, we have
\begin{equation}\label{Ec1}
\sum_{k=0}^\infty k^2 E(k,I_1) = |I_1|+|I_1|^2 
\end{equation}
and
\begin{equation}\label{Ec2}
\sum_{\veck\in\Z_{\geq 0}^2} k_1 k_2 E(\veck,I_1\times I_2) = |I_1\cap I_2|+|I_1|\,|I_2|. 
\end{equation}
The third and higher moments diverge. 

It is important to note that Elkies and McMullen considered the full sequence $\{ \sqrt{n}\bmod 1 : 1\leq n \leq T\}$. Removing the perfect squares $n\in\Box$ does not have any effect on the limit distribution in Theorem  \ref{th:prelim}, since the set of $\alpha$ for which  $\scrN_{T}(I,\alpha)$ is different has vanishing Lebesgue measure as $T\to\infty$. In the case of the second and higher moments, however, the removal of perfect squares will make a difference, and in particular avoid trivial divergences. 

The main result of the present paper is to establish the convergence to the finite mixed moments of the limiting process. The case of the second mixed moment implies, by a standard argument, the convergence of the two-point correlation function stated in Theorem \ref{thm0}, cf.~\cite{EMV_directions_2013}. 
For $I=I_1\times\cdots\times I_m\subset  \R^m$ and $\vecs=(s_1,\ldots,s_m)\in\C^m$ let
\begin{equation}
\M(T,\vecs):=\int_\T (\scrN_{T}(I_1,\alpha)+1)^{s_1} \cdots (\scrN_{T}(I_m,\alpha)+1)^{s_m} \, d\alpha.
\end{equation}

We denote the positive real part of $z\in\C$ by $\re_+(z):=\max\{ \re(z), 0 \}$.

\begin{theorem}\label{th:main0}
Let $I=I_1\times\cdots\times I_m\subset  \R^m$ be a bounded box, and $\lambda$ a Borel probability measure on $\T$ with continuous density. Choose  $\vecs=(s_1,\ldots,s_m)\in\C^m$, such that $\re_+(s_1)+\ldots+\re_+(s_m)<3$. Then,
\begin{equation}\label{limit:main}
\lim_{T\to\infty} \M(T,\vecs) = \sum_{\veck\in\Z_{\geq 0}^m} (k_1+1)^{s_1}\cdots (k_m+1)^{s_m} E(\veck,I).
\end{equation}
\end{theorem}

Our techniques permit to generalize the above results in two ways:

\begin{remark}
Instead of $\scrP_T$ we may consider
\begin{equation}
 \scrP_{T,c}=\{\sqrt n \bmod 1\colon c^2 T< n\le T,\; n\notin\Box\}
 \end{equation}
for any $0\leq c <1$. This setting is already discussed in \cite[Section 3.5]{ElkiesMcM04}, and the upper bounds obtained in the present paper are sufficient to establish Theorem \ref{th:main0} in this case. Note that the limit process is different for each $c$; it coincides with the limit process $E_c(\veck, I)$ studied in \cite{EMV_directions_2013}. As we point out in \cite{EMV_directions_2013}, the second moments of $E_c(\veck, I)$ are Poisson, and hence Theorem \ref{thm0} holds independently of the choice of $c$.
\end{remark}

\begin{remark}
Although Elkies and McMullen assume that $\lambda$ is Lebesgue measure, the equidistribution result that is used to prove Theorem \ref{th:prelim} in fact holds for any Borel probability measure $\lambda$ on $\T$ which is absolutely continuous with respect to Lebesgue measure. This follows from Ratner's theorem by arguments similar to those used by Shah \cite{shah_limit_1996}. It is important to note that the limiting process $E_c(\veck,I)$ will be independent of the choice of $\lambda$. Theorem \ref{th:main0} then follows from the general version of Theorem \ref{th:prelim} for measures $\lambda$ with continuous density (since in this case, for all upper bounds, it is sufficient to restrict the attention to Lebesgue measure). As discussed in \cite{EMV_directions_2013}, the generalization of the above results to $\lambda$ with continuous density yields the convergence of a more general two-point correlation function,
 \begin{equation}
R_N^2(f) = \frac{1}{N} \sum_{m\in\Z} \sum_{\substack{i,j=1\\ i\neq j}}^N f\big(\alpha_i, \alpha_j, N(\alpha_i-\alpha_j+m)\big),
\end{equation}
to the Poisson limit. That is, for all $f\in C_0(\T^2\times\R)$, 
\begin{equation}\label{paircon22}
\lim_{T\to\infty}  R_{N(T)}^2(f) = \int_{\mathclap{\T\times\R}} f(\alpha,\alpha,s)\,d\alpha\, ds.
\end{equation}
\end{remark}

\section{Strategy of proof}

The proof of Theorem \ref{th:main0} follows our strategy in \cite{EMV_directions_2013}. We define the restricted moments
\begin{equation}
\M^{(K)}(T,\vecs):=\int_{\max_j \scrN_{T}(I_j,\alpha)\leq K} (\scrN_{T}(I_1,\alpha)+1)^{s_1} \cdots (\scrN_{T}(I_m,\alpha)+1)^{s_m}  d\alpha .
\end{equation}
Theorem \ref{th:prelim} implies that, for any fixed $K\geq 0$,
\begin{equation}\label{limit:main2}
\lim_{T\to\infty} \M^{(K)}(T,\vecs) = \sum_{\substack{\veck\in\Z_{\geq 0}^m\\ |\veck|\leq K}} (k_1+1)^{s_1}\cdots (k_m+1)^{s_m} E(\veck,I),
\end{equation}
where $|\veck|$ denotes the maximum norm of $\veck$. To prove Theorem \ref{th:main0}, what remains is to show that 
\begin{equation}\label{limit:main3}
\adjustlimits\lim_{K\to\infty} \limsup_{T\to\infty} \left|\M(T,\vecs)-\M^{(K)}(T,\vecs)\right| = 0 .
\end{equation}
To establish the latter, we use the inequality
\begin{equation}
\left|\M(T,\vecs)-\M^{(K)}(T,\vecs)\right|
\leq \int_{\scrN_{T}(\overline I,\alpha)\geq K} (\scrN_{T}(\overline I,\alpha)+1)^{\sigma} d\alpha
\end{equation}
where $\overline I=\cup_j I_j$ and $\sigma=\sum_j \re_+(s_j)$. As in the work of Elkies and McMullen, the integral on the right hand side can be interpreted as an integral over a translate of a non-linear horocycle in the space of affine lattices. The main difference is that now the test function is unbounded, and we require an estimate that guarantees there is no escape of mass as long as $\sigma<3$. This means that \begin{equation}\label{upperb}
\adjustlimits\lim_{K\to\infty} \limsup_{T\to\infty} \int_{\scrN_{T}(\overline I,\alpha)\geq K} (\scrN_{T}(\overline I,\alpha)+1)^{\sigma} d\alpha =0
\end{equation}
implies Theorem \ref{th:main0}. The remainder of this paper is devoted to the proof of \eqref{upperb}.

\section{Escape of mass in the space of lattices\label{sec:affine}}

Let $G=\SL(2,\R)$ and $\Gamma=\SL(2,\Z)$. Define the semi-direct product $G'=G\ltimes \R^2$ by 
\beq (M,\vecxi)(M',\vecxi')=(MM',\vecxi M'+\vecxi'),\eeq
and let $\Gamma'=\Gamma\ltimes \Z^2$ denote the integer points of this group. 
In the following, we will embed $G$ in $G'$ via the homomorphism $M\mapsto (M,\vecnull)$ and identify $G$ with the corresponding subgroup in $G'$.
We will refer to the homogeneous space $\Gamma\quot G$ as the space of lattices and $\Gamma'\quot G'$ as the space of affine lattices. A natural action of $G'$ on $\R^2$ is defined by $\vecx\mapsto \vecx(M,\vecxi):=\vecx M +\vecxi$.

Given an interval $I\subset\R$, define the triangle
\begin{equation}
\fC(I) = \{ (x,y)\in\R^2 : 0< x<2, \; y\in 2 |x| I   \} .
\end{equation}
and set, for $g\in G'$ and any bounded subset $\fC\subset\R^2$,
\begin{equation}
\scrN(g,\fC)= \# ( \fC \cap \Z^2 g) .
\end{equation}
By construction, $\scrN(\,\cdot\,,\fC)$ is a function on the space of affine lattices, $\Gamma'\quot G'$.

Let 
\begin{equation}
\Phi^t =\begin{pmatrix} \e^{-t/2} & 0 \\ 0 & \e^{t/2} \end{pmatrix}, \qquad
\tilde n(u) = \bigg( \begin{pmatrix} 1 & u \\ 0 & 1 \end{pmatrix} , \bigg(\frac{u}{2}, \frac{u^2}{4}\bigg)\bigg).
\end{equation}
Note that $\{ \Phi^t \}_{t\in\R}$ and $\{ \tilde n(u) \}_{u\in\R}$ are one-parameter subgroups of $G'$. Note that $\Gamma' \tilde n(u+2) = \Gamma' \tilde n(u)$ and hence $\Gamma' \{ \tilde n(u) \}_{u\in[-1,1)} \Phi^t$ is a closed orbit in $\Gamma'\quot G'$ for every $t\in\R$.

\begin{lemma}\label{lem:the}
Given an interval $I\subset\R$, there is $T_0>0$ such that for all $T=\e^{t/2}\geq T_0$, $\alpha\in[-\frac12,\frac12]$:
\begin{equation}\label{crude}
\scrN_{T}(I,\alpha) \leq \scrN\left(\tilde n(2\alpha)\Phi^{t}, \fC(I) \right)+ \scrN\left(\tilde n(-2\alpha)\Phi^{t}, \fC(I) \right)
\end{equation}
and, for $-\frac13 T^{-1/2}\leq\alpha\leq \frac13 T^{-1/2}$,  
\begin{equation} \label{zerro}
\scrN_{T}(I,\alpha) = 0.
\end{equation}
\end{lemma}

\begin{proof}
The bound \eqref{crude} follows from the more precise estimates in \cite{ElkiesMcM04}; cf.~also \cite[Sect.~4]{marklof_survey}. The second statement \eqref{zerro} follows from the observation that the distance of $\sqrt n$ to the nearest integer, with $n\leq T$ and not a perfect square, is at least $\frac12 (n+1)^{-1/2}\geq \frac12(T+1)^{-1/2}$.
\end{proof}

A convenient parametrization of $M\in G$ is given by the the Iwasawa decomposition
\begin{equation}\label{iwasawa}
M=n(u)a(v)k(\varphi)=\begin{pmatrix} 1 & u \\ 0 & 1 \end{pmatrix}
\begin{pmatrix} v^{1/2} & 0 \\ 0 & v^{-1/2} \end{pmatrix}
\begin{pmatrix} \cos\phi & -\sin\phi \\ \sin\phi & \cos\phi \end{pmatrix}
\end{equation}
where $\tau=u+\mathrm{i} v$ is in the complex upper half plane $\H=\{ u+\mathrm{i} v \in\C : v>0\}$ and $\phi\in[0,2\pi)$. 
A convenient parametrization of $g\in G'$ is then given by $\H\times [0,2\pi) \times \R^2$ via the decomposition
\begin{equation}\label{ida}
g = (1,\vecxi) n(u)a(v)k(\varphi)\eqqcolon(\tau,\phi;\vecxi).
\end{equation}
In these coordinates, left multiplication on $G$ becomes the group action
\begin{equation}
g\cdot (\tau,\varphi;\vecxi)=(g\tau,\varphi_g;\vecxi g^{-1})
\end{equation}
where for 
\begin{equation}
g =(1,\vecm)\abcd
\end{equation}
we have:
\begin{equation}
g\tau = u_g+\mathrm{i} v_g = \frac{a\tau+b}{c\tau+d}
\end{equation}
and thus
\begin{equation}
v_g=\im (g \tau)=\frac{v}{|c\tau+d|^2};
\end{equation}
furthermore
\begin{equation}
\varphi_g=\varphi+\arg(c\tau+d),
\end{equation}
and
\begin{equation}
\vecxi g^{-1} = (d\xi_1-c\xi_2,-b\xi_1+a\xi_2) - \vecm  .
\end{equation}

We define the abelian subgroups
\[\Gamma_\infty= \left\{\begin{pmatrix}1 &m\\ 0 &1\end{pmatrix}\colon m\in\Z\right\}\subset\Gamma  \]
and
\[\Gamma_\infty'= \left\{\left(\begin{pmatrix}1 &m_1\\ 0 &1\end{pmatrix}, (0,m_2) \right)\colon (m_1,m_2)\in\Z^2\right\}\subset\Gamma' . \]  
These subgroups are the stabilizers of the cusp at $\infty$ of $\Gamma\quot G$ and  $\Gamma'\quot G'$, respectively.

For a fixed real number $\beta$ and a continuous function $f:\R\to\R$ of rapid decay at $\pm\infty$, define the function $F_{R,\beta}\colon \H\times\R^2\to \R$ by 
\begin{equation}
\begin{split}
F_{R,\beta}\left(\tau; \vecxi\right)&=\sum_{\gamma\in\Gamma_\infty\quot \Gamma}\sum_{m\in\Z}f (((\vecxi\gamma^{-1})_1+m)v^{1/2}_\gamma) v^\beta_\gamma\chi_R(v_\gamma) \\
&=\sum_{\gamma\in\Gamma_\infty'\quot \Gamma'} f_\beta(\gamma g) ,
\end{split}
\end{equation}
where $f_\beta:G'\to\R$ is defined by 
\begin{equation}
f_\beta((1,\vecxi) n(u)a(v)k(\varphi)) := f(\xi_1 v^{1/2}) v^\beta \chi_R(v).
\end{equation}
We view $F_{R,\beta}\left(\tau; \vecxi\right)=F_{R,\beta}\left(g\right)$ as a function on $\Gamma'\quot G'$ via the identification \eqref{ida}. 

We show in \cite[Sect.~3]{EMV_directions_2013} that there is a choice of a continuous function $f\geq 0$ with compact support, such that for $\beta=\frac12 \sigma$, and $v\geq R$ with $R$ sufficiently large, we have
\begin{equation}\label{404}
\scrN(g,\fC(I))^\sigma \leq F_{R,\beta}(g)=F_{R,\beta}\left(\tau; \vecxi\right).
\end{equation}

The following proposition establishes under which conditions there is no escape of mass in the equidistribution of translates of non-linear horocycles. In view of Lemma \ref{lem:the} and \eqref{404}, it implies \eqref{upperb} and thus Theorem \ref{th:main0}. (Use $v=1/T$ and note that $\frac{\beta}{2(\beta-1)}>\frac12$ so the choice $\eta=\frac12$ is always permitted.)

\begin{prop}\label{prop:ilya}
Assume $f$ is continuous and has compact support. Let $0\leq \beta <\frac32$. 
Then
\beq 
\adjustlimits\lim_{R\to \infty}\limsup_{v\to 0} \bigg| \int F_{R,\beta}\left(\tilde n(u)a(v)\right) du \bigg|=0
\eeq
where the range of integration is $[-1,1]$ for $\beta<1$, and $[-1,-\theta v^{\eta}]\cup [\theta v^{\eta},1]$ for $\beta\geq 1$ and any $\eta \in [0,\frac{\beta}{2(\beta-1)})$, $\theta\in(0,1)$.
\end{prop}

The proof of this proposition is organized in three parts: the proof for $\beta< 1$, a key lemma, and finally the proof for $1\le \beta< \frac32$. In the following we assume without loss of generality that $f$ is nonnegative,  even, and that $ f(rx) \le f(x)$ for all $r\ge 1$ and all $x\in\R$.

\section{Proof of Proposition \ref{prop:ilya} for $\beta<1$}

(This case is almost identical to the analogous result in \cite{EMV_directions_2013}.) Since $f$ is rapidly decaying and $R\geq 1$, we have
\begin{equation}
F_{R,\beta}\left(\tau; \vecxi\right) \ll_f \overline F_{R,\beta}\left(\tau\right)
\end{equation}
where
\begin{equation}
\overline F_{R,\beta}\left(\tau\right)=\sum_{\gamma\in\Gamma_\infty\quot \Gamma} v^\beta_\gamma\chi_R(v_\gamma).
\end{equation}
Thus
\begin{equation}
\int_{-1}^1 F_{R,\beta}\left(u+\mathrm{i} v;\vecxi \right) du \ll_{f,h}
2 \int_0^1 \overline F_{R,\beta}\left(u+\mathrm{i} v\right) du .
\end{equation}
The evaluation of the integral on the right hand side is well known from the theory of Eisenstein series.
We have
\begin{equation}
\overline F_{R,\beta}\left(\tau\right)=v^\beta \chi_R(v) +
2 \sum_{c=1}^\infty  \sum_{\substack{d=1 \\ \gcd(c,d)=1}}^{c-1} \sum_{m\in\Z} \frac{v^\beta}{c^{2\beta} |\tau+\frac{d}{c}+m|^{2\beta}} \chi_R\left(\frac{v}{c^2 |\tau+\frac{d}{c}+m|^2}\right).
\end{equation}
This function is evidently periodic in $u=\re\tau$ with period one, and its zeroth Fourier coefficient is
(we denote by $\varphi$ Euler's totient function)
\begin{equation}
\int_0^1  \overline F_{R,\beta}\left(u+\mathrm{i} v\right) du    =
v^\beta \chi_R(v) +
2 v^{1-\beta} \sum_{c=1}^\infty  \frac{\varphi(c)}{c^{2\beta}} \int_\R \frac{1}{(t^2+1)^\beta} \chi_R\left(\frac{1}{vc^2 (t^2+1)}\right)  dt.
\end{equation}
The first term vanishes for $v<R$, and the second term is bounded from above by 
\begin{equation}
2 v^{1-\beta} \sum_{c=1}^\infty  \frac{1}{c^{2\beta-1}} \int_\R \frac{1}{(t^2+1)^\beta} \chi_R\left(\frac{1}{vc^2 (t^2+1)}\right) \, dt = 2 v^{1/2} \sum_{c=1}^\infty K_R(cv^{1/2})
\end{equation}
with the function $K:\R_{>0}\to \R_{\geq 0}$ defined by
\begin{equation}
K_R(x)=  \frac{1}{x^{2\beta-1}} \int_\R \frac{1}{(t^2+1)^\beta} \chi_R\left(\frac{1}{x^2 (t^2+1)}\right) \, dt .
\end{equation}
We have $K_R(x)\ll \max\{ 1, x^{2\beta-1}\} \leq \max\{ 1, x^{-1/2}\}$ and furthermore $K_R(x)=0$ if $x>R^{-1/2}$. Thus
\begin{equation}
\lim_{v\to 0} v^{1/2} \sum_{c=1}^\infty K_R(cv^{1/2}) = \int_\R K_R(x) dx,
\end{equation}
which evaluates to a constant times $R^{-(1-\beta)}$. \qed

\section{Key lemma}

\begin{lemma}\label{lem:daniel} Let $f\in C(\R)$ be rapidly decreasing and let 
\beq S=\sum_{\substack{D \le c \le 2D \\ 1 \le d \le D \\ \gcd(c,d)=1}} \sum_{m \in \Z} f \left(T \left(\frac {d^2}{4c} + m \right)\right).\eeq
Then, for $D \ge 1, T > 1$ and any $\epsilon > 0$, we have 
\beq  S\ll  \frac {D^{2}}{T^{1-\eps}},   \label{eq:daniel}\eeq
where the implied constant depends only on $\epsilon$ and $f$.
\end{lemma}

\begin{proof} 
We assume without loss of generality that $f$ is even, non-negative, and of Schwartz class.
We prove two statements about $S$ from which the statement of the Lemma will follow. They are 
\beq\label{eq:statement1}S\ll \frac{D^{2+\eps'}}{T} \qquad \text{for any $\epsilon'>0$} \eeq
and
\beq\label{eq:statement2}S\ll \frac{D^2}{T}+D^{3/2}T^{\eps''} \qquad \text{for any $\epsilon''>0$.} \eeq
Then $S$ is bounded by the smaller of these expressions, and it is easy to see that the bound in \eqref{eq:daniel} holds no matter which realizes the minimum. 

Equation \eqref{eq:statement1} is verified by summing over quadratic residues modulo $4c$. Note that the conditions $1 \le d \le c$ and $\gcd(c, d) = 1$ imply $\frac {d^2}{4c} \notin \Z$. 
For coprime $D \le c \le 2D$ and $1 \le d \le D$ and $m \in \Z$ such that $\left| \cramped{\frac {d^2}{4c} + m} \right| \ge \frac12$, we use rapid decay of $f$ to get 
\begin{align} 
\sum_{\substack{D \le c \le 2D \\ 1 \le d \le D \\ \gcd(c,d)=1}} 
\sum_{\substack{m \in \Z \\ \left| \cramped[\scriptstyle]{\frac {d^2}{4c} + m} \right| \ge \frac12}} f 
\left(T \left(\frac {d^2}{4c} + m \right)\right) 
&\ll \sum_{\substack{D \le c 
\le 2D \\ 1 \le d \le D}} \sum_{\substack{m \in \Z \\ \left| \frac {d^2}{4c} + 
m \right| \ge \frac12}} \frac 1{\left|T \left( \frac {d^2}{4c}+m \right)\right|^A} \\
&\ll \frac 1{T^A} \sum_{\substack{D \le c \le 2D \\ 1 \le d \le D}} \sum_{m \in \Z \setminus \{0\}} \frac 1{|m|^A} \\
&\ll \label{goodm}\frac {D^2}{T^A}
\end{align}
for every $A > 1$.

For a positive integer $n$, denote by $\omega(n)$ the number of distinct prime factors of $n$ and by $\tau(n)$ the number of divisors of $n$.

For coprime $D \le c \le 2D$ and $1 \le d \le D$ and $m \in \Z$ such that $\left| \cramped{\frac {d^2}{4c} + m }\right| < \frac12$, we have (denote by $\| \cdot \|$ the distance to the nearest integer)
\begin{align}
\sum_{\substack{D \le c \le 2D \\ 1 \le d \le D \\ \gcd(c,d)=1}} 
\sum_{\substack{m \in \Z \\ \left| \cramped[\scriptstyle]{ \frac {d^2}{4c} + m} \right| < \frac12}} f 
\left(T \left(\frac {d^2}{4c} + m \right)\right) 
& \le  \sum_{\substack{D \le c \le 2D \\ 1 \le d \le D \\ \gcd(c,d)=1}} f \left(T \left\| \frac {d^2}{4c}  \right\| \right) \\
& \ll  \sum_{\substack{D \le c \le 2D \\ 1 \le j \le 2c \\ j \text{ is a square mod }4c}} 2^{\omega(4c)} f\left(T \frac j{4c} \right),
\end{align}
since \beq\#\{ d \bmod 4c \, : \, d^2 \equiv j \bmod {4c},\; \gcd(c,d)=1 \} \ll 2^{\omega(4c)}\eeq for every $j$.
Now $2^{\omega(4c)}$ is the number of squarefree divisors of $4c$ and is therefore at most $\tau(4c)$. Combined with rapid decay (we use $f(t) \ll \frac 1t$), this yields
\beq \ll \frac 1T \sum_{\substack{D \le c \le 2D \\ 1 \le j \le 4D}} \tau(4c) \frac cj.\eeq
Finally the fact that for every $\epsilon''' > 0, \, \tau(n) \ll n^{\epsilon'''}$ gives
\beq \ll \frac {D^{2+\epsilon'''} \log D}T \eeq 
which is 
\beq\label{badm} \ll \frac {D^{2+\epsilon'}}T \eeq for every $\epsilon' > 0$. It suffices to note that $(\ref{goodm}) \ll (\ref{badm})$ to verify \eqref{eq:statement1}.

Inequality \eqref{eq:statement2} is obtained as follows. The Poisson summation formula yields 
\beq S \ll c_0D^2 + \bigg| \sum_{\substack{n \ne 0 \\ D \le c \le 2D \\ 1 \le d \le 4c}} c_n e\left(\frac{ n d^2}{4c}\right) \bigg|.\label{eq:poisson}\eeq
Here \beq c_n = \frac 1T \hat{f} \left( \frac nT \right) \eeq where $\hat{f}$ is the Fourier transform of $f$, which is also of Schwartz class, and $e(z)=e^{2\pi \mathrm{i} z}$ is the usual shorthand. The second term in \eqref{eq:poisson} is bounded by
\begin{align}
\bigg|\sum_{\substack{n \ne 0 \\ D \le c \le 2D \\ 1 \le d \le 4c}} c_n e\left(\frac{ n d^2}{4c}\right) \bigg|
&\le \sum_{r=1}^{8D} \sum_{n\ne 0}\sum_{\substack{{4D/r \le c \le 8D/r}\\{ \gcd(n,4c)=1}}} \bigg| c_{nr} r \sum_{d\bmod{4c}}e\left(\frac{nd^2}{4c} \right) \bigg|\\
&\le \sum_{n \ne 0} \sum_{1 \le r \le 8D} \sum_{4D/r \le c \le 8D/r} |c_{nr}| r\sqrt{8c} . \label{Gauss} 
\end{align} 
The last inequality follows from the well known evaluation the classical Gauss sum with $\gcd(n,4c)=1$
\beq \label{Gaussvalue} \sum_{d \bmod {4c}}  e \left( \frac {nd^2}{4c} \right) = (1+\mathrm{i})\; \epsilon_n^{-1} \left( \frac {4c}{n} \right) \sqrt{4c},  \eeq
where $\left( \frac {4c}n \right)$ is the Jacobi symbol and $\epsilon_n = 1$ or $\mathrm{i}$ if $n=1$ or $3\bmod 4$, respectively.

When $|nr| < T$ we use the fact that the Fourier transform of $f$ is bounded:
\beq |c_{nr}| \ll \frac 1T. \eeq
Therefore,
\beq \eqref{Gauss} \ll \frac {D^{3/2}}T \sum_{\substack{n \ne 0 \\ 1 \le r \le 8D \\ |nr| < T}} \frac 1{\sqrt{r}} = \frac {D^{3/2}}T \sum_{1 \le r \le 8D} \frac 1{\sqrt{r}} \sum_{|n| < T/r} 1 \ll D^{3/2} \eeq 

When $|nr| > T$, we use the fact that the Fourier transform of $f$ decays faster than any polynomial since $f$ is smooth:
\beq |c_{nr}| \ll \frac 1{|n|^A r^A} T^{A-1}. \eeq 
We take $A>1$. 
Then we have 
\begin{align}
\eqref{Gauss} & \ll D^{3/2}T^{A-1}\sum_{\substack{|nr|>T\\n\ne0\\1\le r\le 8D}}\frac{1}{|n|^Ar^{A+1/2}} \ll D^{3/2}T^{A-1}\sum_{\substack{nr>T\\ n,r\ge 1}}\frac1{(nr)^A}
\\
&\ll D^{3/2}T^{A-1}\sum_{k=T}^\infty k^{-A+\eps''}\ll D^{3/2}T^{\eps''}.
\end{align}
This proves \eqref{eq:statement2} and the Lemma.

%
%
\end{proof}

\section{Proof of Proposition \ref{prop:ilya} for $\beta\geq 1$}

We have
\begin{equation}
\tilde n(u)a(v) = (u+\mathrm{i} v, 0 ; \vecxi) 
\end{equation}
where $\vecxi=(u/2,-u^2/4)$. For this choice we have
\begin{align}\label{eq:firstterm}
F_{R,\beta}(\tau; \vecxi)&=
2\sum_{m\in\Z}f\left((m+u^2/4) \frac{v^{1/2}}{|\tau|}\right)\frac{v^\beta}{|\tau|^{2\beta}}\chi_R\left(\frac{v}{|\tau|^2}\right)+\\
&+2\sum_{\substack{{(c,d)\in\Z^2}\\{\gcd(c,d)=1}\\{c>0,d\ne 0}}}\sum_{m\in\Z}f\left((cu^2/4+du/2+m)\frac{v^{1/2}}{|c\tau+d|}\right)
\frac{v^{\beta}}{|c\tau+d|^{2\beta}}\chi_R\left(\frac v{|c\tau+d|^2}\right).\label{eq:secondterm}
\end{align}
The integral of the first term 
tends to zero as $v\to0$. We write $T=\frac{v^{1/2}}{|\tau|}=\frac{v^{1/2}}{\sqrt{u^2+v^2}}.$ Indeed, for $m\neq 0$ we have $|f((m+u^2/4)T)|\ll (|m|T)^{-A}$ from rapid decay, so that 
\beq
\int_J \eqref{eq:firstterm}du\ll \int_{-1}^1 T^{-A+\beta} du\ll  \int_{-1}^1 \left(\frac{v^{1/2}}v\right)^{-A+\beta} du \ll v^{\frac{A-\beta}2 },
\eeq
where $J=[-1,-\theta v^{\eta}]\cup [\theta v^{\eta},1]$.
If $A>\beta$, then this contribution is negligible as $v\to 0$. For $m=0$, we have $|f|\ll 1$ so that the contribution of this term is, assuming $\beta>1$,
\beq
\ll \int_J  T^{\beta}du \ll \int_{\theta v^{\eta}}^1 u^{-\beta}   v^{\beta/2}du = \frac{\theta^{1-\beta}}{\beta-1} \, v^{\beta/2+\eta-\beta \eta} \to 0
\eeq
since $\eta<\frac{\beta}{2(\beta-1)}$; similarly for $\beta=1$.

It remains to analyze the contribution of \eqref{eq:secondterm}. 
Notice that the this term is nonzero only when $-d/c$ is in the range of integration for $u$, which is contained in the interval $[-1,1]$. Therefore we restrict the summation to $0<|d|\le c$.
Now we perform the substitution $t=\left(u+d/c\right)v^{-1}$ to ``zoom in'' on each rational point and extend the range of integration to all of \R. This gives 
\begin{multline}\int_{t=-\infty}^\infty f\left(\left(\tfrac c4\left(-\tfrac dc +tv\right)^2+\tfrac d2\left(-\tfrac dc+tv\right)+m \right)\frac1{\sqrt{c^2 v (t^2+1)}}\right)
\times
\\
\times
\frac v{(c^2v(t^2+1))^\beta}\chi_R\left(\frac1{c^2v(t^2+1)}\right)  dt\end{multline}
and we need to bound 
\beq\sum_{c=1}^\infty\int_{t\in\R}\sum_{\substack{{0<|d|\le c}\\{(c,d)=1}}}\sum_{m\in\Z} f\!\left(\left(-\tfrac {d^2}{4c}+m + O(ctv)\right)\!\frac1{\sqrt{c^2 v (t^2+1)}}\right) \frac  {v\, dt}{(c^2v(t^2+1))^\beta}\chi_R\!\left(\frac1{c^2v(t^2+1)}\right)\!.\label{eq:tobound}\eeq
Now we decompose the region $\dfrac 1{\sqrt{c^2v(t^2+1)}} \ge \sqrt{R}$ into dyadic regions \[2^j \le \dfrac 1{\sqrt{c^2 v (t^2+1)}} < 2^{j+1}\] for $j \gg \log R$.
We can thus bound \eqref{eq:tobound} by 
\begin{multline*}\sum_{j \gg \log R} \sum_{c \ge 1} \int_\R \sum_{\substack{{0<|d|\le c}\\{(c,d)=1}}}\sum_{m\in\Z} f\left(\left(-\tfrac {d^2}{4c}+m + O(ctv)\right)\frac1{\sqrt{c^2 v (t^2+1)}}\right) \times \\\times \frac  v{(c^2v(t^2+1))^\beta}\chi_{[2^j, 2^{j+1})} \left(\frac1{\sqrt{c^2v(t^2+1)}}\right)  dt \end{multline*}
\begin{multline*} 
\le\sum_{j \gg \log R} \sum_{c \ge 1} \int_\R \sum_{\substack{{0<|d|\le c}\\{(c,d)=1}}}\sum_{m\in\Z} f\left(2^j\left(-\tfrac {d^2}{4c}+m + O(ctv)\right)\right)\times\\\times \frac  v{(c^2v(t^2+1))^\beta}\chi_{[2^j, 2^{j+1})} \left(\frac1{\sqrt{c^2v(t^2+1)}}\right) dt 
\end{multline*}
\begin{align}
\label{eq:doublesum}&\ll v \sum_{j \gg \log R} 2^{2 \beta j} \int_\R   \sum_{\frac {2^{-(j+1)}}{\sqrt{v(t^2+1)}} \le c \le \frac {2^{-j}}{\sqrt{v(t^2+1)}}} \sum_{\substack{{0<|d| \le \frac {2^{-j}}{\sqrt{v(t^2+1)}}}\\{(c,d)=1}}} \sum_{m\in\Z}  f\left(2^j\left(-\tfrac {d^2}{4c}+m + O(ctv)\right)\right) dt .\end{align}

It remains to remove the error term $O(ctv)$ from the argument of $f$ to apply Lemma \ref{lem:daniel}. We have that  $2^j ctv\ll v^{1/2}$ for every $j$. Define 
\[f^*(x)=\max_{-1\le y\le 1} f(x+y).\] 
Then, $f^*(x)\ge f(x+O(v^{1/2}))$ for $v$ small enough, and we can bound \eqref{eq:doublesum} by a similar expression with $ f^*\left(2^j\left(-\tfrac {d^2}{4c}+m \right)\right) $ in place of $ f\left(2^j\left(-\tfrac {d^2}{4c}+m + O(ctv)\right)\right) $. To bound this we apply Lemma \ref{lem:daniel} with $D \sim \frac {2^{-(j+1)}}{\sqrt{v(t^2+1)}}$, $T = 2^j$, and $\eps=\frac32-\beta>0.$ Then we have
\begin{align}\eqref{eq:doublesum}& \ll v\sum_{j\gg \log R} 2^{2\beta j}\int_{t\in \R}\frac{D^2}{T^{1-\eps}} dt\\
&\ll \sum_{j\gg \log R}2^{j(2\beta-3+\eps)}=\sum_{j\gg \log R} 2^{j(\beta-3/2)}\to 0
\end{align}
as $R \to \infty$ by our choice of $\eps.$ \qed

\parindent=0pt
\small

\bibliographystyle{plain}
\bibliography{bibliography}

\footnotesize
\parindent=0pt

\textsc{Daniel El-Baz, School of Mathematics, University of Bristol, Bristol BS8~1TW, U.K.} \texttt{daniel.el-baz@brisol.ac.uk}
\smallskip

\textsc{Jens Marklof, School of Mathematics, University of Bristol, Bristol BS8~1TW, U.K.} \texttt{j.marklof@bristol.ac.uk}
\smallskip

\textsc{Ilya Vinogradov, School of Mathematics, University of Bristol, Bristol BS8~1TW, U.K.} \texttt{ilya.vinogradov@bristol.ac.uk}

\end{document}